\documentclass[12pt, reqno]{article}

\usepackage{fullpage}

\usepackage{graphicx}
\usepackage{caption}
\usepackage{subcaption}

\usepackage{amsmath,amsfonts,amssymb,graphics,amsthm}

\usepackage{mdframed} 

\usepackage[normalem]{ulem} 

\usepackage{hyperref}
\hypersetup{
    colorlinks=true,
    linkcolor=blue,
    citecolor=red,
    urlcolor=blue,
    pdfborder={0 0 0}
}            

\usepackage{graphicx}
\usepackage{xcolor}

\usepackage[font=sf, labelfont={sf,bf}, margin=1cm]{caption}

\usepackage{cleveref}
  \crefname{theorem}{Theorem}{Theorems}
  \crefname{thm}{Theorem}{Theorems}
  \crefname{lemma}{Lemma}{Lemmas}
  \crefname{lem}{Lemma}{Lemmas}
  \crefname{remark}{Remark}{Remarks}
  \crefname{prop}{Proposition}{Propositions}
  
  \crefname{defn}{Definition}{Definitions}
  \crefname{corollary}{Corollary}{Corollaries}
  \crefname{section}{Section}{Sections}
  \crefname{figure}{Figure}{Figures}

\newtheorem{thm}{Theorem}[section]
\newtheorem{lemma}[thm]{Lemma}
\newtheorem{corollary}[thm]{Corollary}

\numberwithin{equation}{section}

\theoremstyle{definition}
\newtheorem{remark}[thm]{Remark}

\def\cP{\mathcal{P}}

\def\E{\mathbb{E}}
\def\C{\mathbb{C}}
\def\H{\mathbb{H}}
\def\R{\mathbb{R}}

\def  \p- {p\textunderscore}
\def \ka{\kappa}

\def\hf{\hat{f}}

\DeclareMathOperator{\Imag}{Im}

\usepackage{tikz}

\newcommand {\eqd} {\stackrel{(d)}{=}}
\title{On the regularity of SLE\ trace}
\author{P.K. Friz, H. Tran \\
TU and WIAS\ Berlin, UCLA}

\begin{document}
\maketitle

%
%
%
%
%
%
%

\begin{abstract}
We revisit regularity of SLE trace, for all $\kappa \neq 8$, and establish Besov regularity under the usual half-space capacity parametrization.
With an embedding theorem of Garsia--Rodemich--Rumsey type, we obtain finite moments (and hence almost surely) optimal variation
regularity with index $\min (1 + \kappa / 8, 2) $, improving on previous works of Werness, and also (optimal) H\"older
regularity \`a la Johansson Viklund and Lawler. 
\end{abstract}

\section{Introduction}

It is classical that SLE$_{\kappa }$ has a.s. continuous trace $\gamma$, any $\kappa
\in (0,\infty )$. (The trivial case $\kappa =0$ will be disregarded throughout.) With the exception of $\kappa =8$, the (classical) proof
has two steps: (1)\ estimates of moments of $\hat{f}_{t}^{\prime }\left(
iy\right) $ the derivative of the shifted inverse (Loewner) flow (2)
partition of $\left( t,y\right) $-space into Whitney-type boxes, together
with a Borel--Cantilli argument.  This strategy of proof is very standard (see \cite{RS05, Law05} or also \cite{JVRW14})
and was 
subsequently refined in \cite{JVL11} and \cite{Lind08} 
to show that SLE$_{\kappa }$ (always in half-space
parametrization) is H\"{o}lder continuous on any compact set in $(0,\infty )$
with any H\"{o}lder exponent less than%
\begin{equation*}
\alpha _{\ast }\left( \kappa \right) =1-\frac{\kappa }{24+2\kappa -8\sqrt{%
\kappa +8}};
\end{equation*}%
on compact sets in $[0,\infty )$, the critical exponent has to be modified to 
$$\alpha _{0}\left(
\kappa \right) =\min \left( \alpha _{\ast }\left( \kappa \right) ,1/2\right).$$ 

On the other hand, based on Aizenman--Burchard techniques, it was shown
in \cite{Wer12}, under the (technical) condition $\kappa \le 4$, that SLE enjoys $p$-variation
regularity for any
\begin{equation*}
p < p_{\ast }=1+\frac{\kappa }{8}. 
\end{equation*}%
Loosely stated, the main result of this paper is a Besov regularity for SLE of the form

\begin{thm} \label{thm:preBesovAintro}
Assume $\kappa > 0$, and fix $T>0$. Then, for suitable $\delta \in (0,1), q >1$ we have
\[  \mathbb{E}\left\Vert \gamma \right\Vert _{W^{\delta ,q};\left[ 0,T\right] }^{q}  <\infty. 
\]
In particular, for a.e. realization of SLE trace, we have $\gamma(\omega)|_{[0,T]} \in W^{\delta,q}$.
\end{thm}

For the (classical) definition of the Besov space $W^{\delta ,q}$, see (\ref{equ:sob}) below; we also postpone the important precise description of possible values $\delta, q$. 
In the spirit of step (1), cf. the discussion at the very beginning of the introduction, moment estimates on $\hat{f}_{t}^{\prime }\left(iy\right)$, taken in the sharp form of
 \cite{JVL11}, are an important ingredient in establishing Theorem \ref{thm:preBesovAintro}. In turn, this theorem unifies and extends previous works \cite{JVL11, Wer12} by exhibiting Besov regularity as common source of both (optimal) H\"older {\it and} $p$-variation regularity for SLE trace. 
 More specifically, we have the following first corollary which recovers the optimal H\"older exponent of SLE trace, as previously established in \cite{JVL11} and in fact improves their almost-sure statement to finite moments. 


\begin{corollary} Assume $\kappa \ne 8$. On compact sets in $[0,\infty)$ (resp. $(0,\infty)$), the $\alpha$-H\"older norm of SLE curve
with $\alpha < \alpha_0 $ (resp. $\alpha<\alpha_*$) has finite $q$-moment, for some $q>1$. 
\end{corollary}
The following extends \cite{Wer12} (from $\kappa \le 4$) to all $\kappa \ne 8$, and again improves from a.s.-finiteness to existence of moments. At this paper neared completion, we learned that Lawler and Werness \cite{LW16} found an independent proof of a.s. finite p-variation of SLE trace (for $\kappa<8$).%
\begin{corollary} Assume $\kappa \ne 8$. On compact sets in $[0,\infty)$ the $p$-variation norm of SLE curve with $p > p_*$ has finite $q$-moment, for some $q>1$. 

\end{corollary}

As a further corollary, precise statement left to Corollary \ref{cor:dim}, we note that our $p$-variation result implies an upper bound on the Hausdorff dimension of the SLE trace; along the lines of the original Rohde--Schramm paper. 
Our approach should also allow to bypass step (2), cf. the discussion in the very beginning, in proving continuity of SLE trace, but we will not focus on this aspect here. 
\medskip

We give some discussion about the basic ideas.
We note that every $\alpha $-H\"{o}lder continuous curve is
automatically of finite $p$-variation, with $p=1/\alpha $. In some prominent
cases, this yields the correct (optimal) $p$-variation regularity: for
instance, Brownian motion is $\left( 1/2-\varepsilon \right) $-H\"{o}lder
and then of finite $\left( 2+\varepsilon \right) $-variation, $\varepsilon >0
$. However, already for elements in the Cameron--Martin space $W^{1,2}\left[
0,1\right] $, that is, absolutely continuous curves $h$ with $\int_{0}^{1}|%
\dot{h}|^{2}dt<\infty $, this fails: in general, such an $h$ is $1/2$-H\"{o}%
lder (hence automatically of $2$-variation) but in fact has finite $1$%
-variation. The same phenomena is seen for SLE,%
\begin{equation*}
p_{\ast }<\frac{1}{\alpha _{\ast }}\leq \frac{1}{\alpha _{0}}.
\end{equation*}%
Let us also note that, from the stochastic\ point of view, the half-space
parametrization is not fully satisfactory as it induces an artifical,
directed view on SLE. A decisive advantage of $p$-variation is its
invariance with respect to reparametrization, related at least in spirit to
the natural parametrization introduced in \cite{LR15}, \cite{LSh11}. The above
Cameron--Martin example in fact holds a key message: Sobolev-regularity is
ideally suited to guarantee both $\alpha $-H\"{o}lder and $p$-variation
regularity with $p<1/\alpha $. More quantitatively, the (elementary)
embedding 
\begin{equation*}
W^{1,2}\subset C^{1\text{-var}}\cap C^{1/2\text{-H\"{o}l}}
\end{equation*}%
(always on compacts sets in $[0,\infty)$) has been generalized in \cite{FV06}, by a delicate application of the Garsia-Rodemich-Rumsey inequality,
 to Besov spaces as follows:

\begin{thm}[Besov-variation embedding] \label{thm:BesovVariation}
Assume $\delta  \in \left( 0,1\right) ,q\in \left( 1,\infty \right)$
 such that $\delta -1/q>0$. Set $p := 1/\delta$, $\alpha :=\delta- 1/q$.
Then there exists a constant $C$, such that for all $0\leq s<t < \infty$,%
\begin{equation}\label{lem:variation}
\left\Vert x\right\Vert _{ p \text{-var;}\left[ s,t\right] }\leq C\left\vert t-s\right\vert ^{\alpha}
\left\Vert x\right\Vert _{W^{\delta ,q};\left[ s,t\right] }.
\end{equation}
\end{thm}
\noindent The estimates holds for arbitrary continuous paths $x$, even with values in general metric spaces; for us, of course,
$x$ takes values in $\C$. Note that the left-hand side dominates the increment $x_{s,t}$, so that the following (classical) Besov-H\"older embedding
appears as immediate consequence,
\begin{equation}\label{lem:Holder}
\left\Vert x\right\Vert _{\alpha%
\text{-H\"ol;}\left[ s,t\right] }\leq C 
\left\Vert x\right\Vert _{W^{\delta ,q};\left[ s,t\right] }.
\end{equation}

%
%
%

The point of Theorem \ref{thm:BesovVariation} is the ``gain'' $p<\frac{%
1}{\alpha }$, which can be substantial for integrability parameter $%
q<<\infty $. For instance $W^{\delta ,2}$ $\left( q=2\right) $ is closely
related to the Cameron--Martin space of fractional Brownian motion (fBm) in
the rough regime with Hurst parameter $H\in (0,1/2]$; in this case $\delta
=H+1/2$ and the above implies that all such paths fall into the reign of
Young integration (which requires $p$-variation with $p<2$), which is
certainly not implied by $\left( H-\varepsilon \right) $-H\"{o}lder
regularity of such Cameron--Martin paths. (This was a crucial ingredient in
the development of Malliavin calculus for rough differential equations
driven by fBm, see e.g. \cite{CF10, CHLT15}.) 

On the other hand, the gain  $\frac{1}{\alpha }-p=O\left( 1/q\right) $
vanishes when $q\uparrow \infty $, which is exactly the reason why Brownian
motion, which has $q$-moments for all $q<\infty $, has $p$-variation
regularity no better than what is implied by its $\alpha $-H\"{o}lder
regularity. This, however, is not the case for SLE and our starting point is
precisely to estimate $q$-moments for increments of SLE\ curves which in
turn leads to a.s. $W^{\delta ,q}$-regularity (and actually some finite
moments of these Besov-norms). By careful booking-keeping, and optimizing
over, the possible choices of $\delta ,q$ (for given $\kappa $) we then
obtain the desired variation and H\"{o}lder regularity of SLE. Surprisingly
perhaps, the intricate correlation structure of SLE\ plays almost no role
here, the entire regularity proof is channeled through knowledge of moments
of the increments of the curve (very much in the spirit of Kolmogorov's
criterion\footnote{%
... which in fact is little more than the embedding $W^{\delta ,q}\subset
C^{\alpha \text{-H\"{o}l}}$.}). At last, our work suggests a viable new route towards (an analytic proof) for existence of SLE trace when $\kappa
=8$, for the Garsia--Rodemich--Rumsey based proof in \cite{FV06} offers the 
flexibility to go beyond the Besov scale and e.g. allows to deal with logarithmic modulus, to be pursued elsewhere.

{\bf{Acknowledgement:}} PKF acknowledges financial support from the European Research Council (ERC) through a Consolidator Grant, nr. 683164. HT acknowledges partial support by NSF grant DMS-1162471.

\section{Moments of the derivative of the inverse flow}
Fix $\kappa\in (0,\infty)$. Let $U_t=\sqrt{\kappa}B_t$, where $B$ is a standard Brownian motion. Let $(g_t)$ be the downward flow SLE, that is, solutions to
$$\partial_t g_t(z)=\frac{2}{g_t(z)-U_t},~~~ g_0(z)=z ~~~\mbox{ for } z\in \mathbb{H},$$
 and let $f_t=g^{-1}_t$ and $\hat{f}_t(z)=f_t(z+U_t)$. Let $\gamma$ be the SLE$_\kappa$ curve. It follows from \cite{RS05} and \cite{LSW04} that a.s. for all $t\geq 0$,
$$\gamma(t)=\lim_{u\to 0^+} \hat{f}_t(iu) .$$
Suppose
$$-\infty <r< r_c:=\frac{1}{2} + \frac{4}{\ka}$$
$$q:=q(r) = r\left(1+\frac{\ka}{4}\right) - \frac{\ka r^2}{8}$$
$$\zeta:=\zeta(r)=r-\frac{\ka r^2}{8}.$$
Through out this note, we always assume $r<r_c$. Note that $q$ is strictly increasing with $r$ on an interval which contains $(-\infty,r_c)$. The following moment estimate will be important. 
\begin{lemma}\label{lem:der} 
There exists a constant $c<\infty$ depending on $r$ such that for all $s,y\in (0,1]$,
$$
\E(|\hat{f}'_s(iy)|^q) \leq \left\{\begin{array}{rl}
c s^{-\zeta/2} y^{\zeta} &\mbox{ when } s\geq y^2, \\
c A_s y^\zeta & \mbox{\ in general ,}\nonumber
\end{array} \right.
$$

where $A_s= \max(s^{-\zeta/2},1)$.
\end{lemma}
\noindent
This is just a corollary of \cite[Lemma 4.1]{JVL11} in which they prove that for all $t\geq 1$,
$$\E[|\hat{f}'_{t^2}(i)|^q]\leq c t^{-\zeta}.$$
\noindent
When $t\in (0,1]$, the Koebe distortion theorem implies there is a constant $c$ such that 
$$|\hat{f}'_{t^2}(i)|\leq c.$$
\noindent
By the scaling property of SLE
$$\hat{f}'_s(iy)\eqd \hat{f}'_{s/y^2}(i),$$
hence
$$
\E(|\hat{f}'_s(iy)|^q) = \E(|\hat{f}'_{s/y^2}(i)|^q)\leq \left\{
\begin{array}{rcl}
c s^{-\zeta/2} y^\zeta& \mbox{ when }& s\geq y^2 \\
c \leq c s^{-\zeta/2} y^\zeta & \mbox{ when } & s\leq y^2 \mbox{ and } \zeta>0\\
c \leq c  y^\zeta &\mbox{ when } & s\leq y^2 \mbox{ and } \zeta\leq 0.
\end{array}
\right.
$$

\noindent
We also make use of the following two lemmas.
\begin{lemma}\cite[Lemma 3.5]{JVL11}
\label{key-lemma2}
If $0\leq t-s\leq y^2$ where $y=\Imag(z)$, then
$$|f_t(z)-f_s(z)|\lesssim y |f'_s(z)|.$$
\end{lemma}
\begin{lemma}\cite[Exercise 4]{L12}\label{key-lemma3}
There exist $C>0$ and $l>1$ such that if $h:\H\to \C$ is a conformal transformation, then for all $x\in \R$, $y>0$,
$$|h'(xy+iy)|\leq C(x^2+1)^l |h'(iy)|.$$
\end{lemma}
\section{Moment estimates for SLE increments} 

We prepare our statement with defining some set of ``suitable'' $r$'s. Unless otherwise stated, we always assume $\kappa \in (0,\infty)$.
\begin{eqnarray*}
I_{0} &:=& I_{0}\left( \kappa \right) :=\left\{ r\in \mathbb{R}%
:r<r_{c}\right\} \text{ with }r_{c}\equiv \frac{1}{2}+\frac{4}{\kappa }, \\
I_{1} &:=& I_{1}\left( \kappa \right) :=\left\{ r\in \mathbb{R}:q>1\right\} 
\text{ with }q=q\left( r\right) =\left( 1+\frac{\kappa }{4}\right) r-\frac{%
\kappa r^{2}}{8}, \\
I_{2} &:=& I_{2}\left( \kappa \right) :=\left\{ r\in \mathbb{R}:q+\zeta
>0\right\} \text{ with }\zeta =\zeta \left( r\right) =  r- \frac{\kappa r^{2}}{8}.
\end{eqnarray*}

\begin{lemma} \label{lem:I1I2I3}
One has $I_{1}=\left( r_{1-},r_{1+}\right) $ with $r_{1\pm }=\frac{\left(
4+\kappa \right) \pm \sqrt{\kappa ^{2}+16}}{\kappa }$ and 
\begin{equation}
0<r_{1-}<r_{c}<r_{1+}.  \label{r1mlerc}
\end{equation}%
Moreover, $I_{2}=\left( 0,2r_{c}\right) ,$ so that%
\begin{equation} \label{def:I}
I := I (\kappa)  := I_{0}\cap I_{1}=I_{0}\cap I_{1}\cap I_{2}=\left( r_{1-},r_{c}\right).
\end{equation}
\end{lemma}

\begin{proof}
Solving a quadratic equation, we see that $I_{1}=\left\{ r:q>1\right\} $ is
of the given form $\left( r_{1-},r_{1+}\right) $. Direct inspection of $%
q\left( r_{c}\right) >1$ implies (\ref{r1mlerc}). At last, $I_{2}$ $\ $is
given by those $r$ for which $\left( 1+\frac{\kappa }{4}\right) r-\frac{%
\kappa r^{2}}{8}+r- \frac{\kappa r^{2}}{8}=r\left( 2+\frac{\kappa }{4}-\frac{%
\kappa }{4}r\right) >0$. It follows that $I_{2}=\left( 0,\frac{4}{\kappa }%
\left( 2+\frac{\kappa }{4}\right) \right) =\left( 0,2r_{c}\right) $.
\end{proof}

\begin{figure}
\centering
\begin{subfigure}{.5\textwidth}
  \centering
  \includegraphics[width=0.8\linewidth]{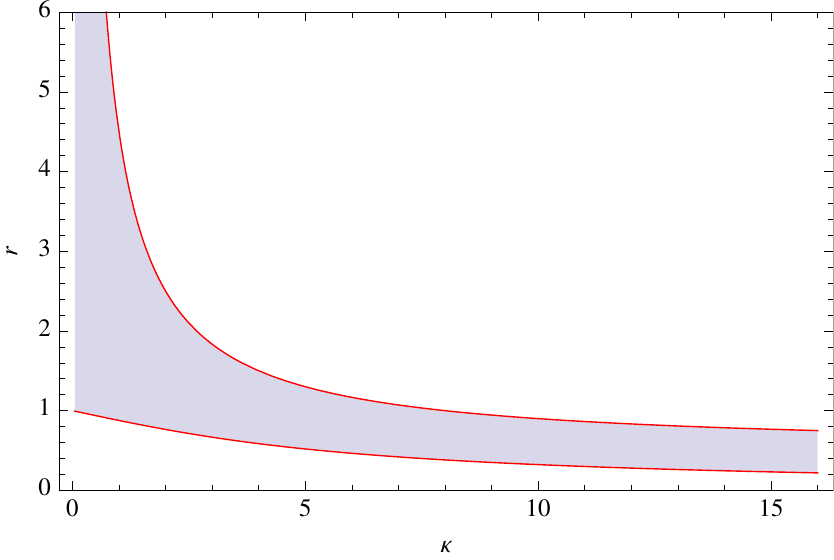}
  \label{fig:sub2}
\end{subfigure}
\caption{Admissible $r$ in the sense of $r \in I$, as function of $\kappa$} 
\label{fig:test}
\end{figure}

\begin{lemma} \label{l:gts}
Let $r \in I = I(\kappa)$, as defined in \eqref{def:I}. 
Then, for any $0<s\leq t\leq 1$,
$$\E[|\gamma(t)-\gamma(s)|^q]  \leq C(t-s)^{(q + \zeta)/2} (A_s+t^{-\zeta/2}) + C (t-s)^{\frac{1}{2}(q+\tilde{\zeta}/\theta)}t^{-{\color{red}\tilde{\zeta}}/(2\theta)},$$
where $\tilde{r}\in (r,r_c)$ arbitrarily, $\tilde{q} = q(\tilde{r}), \tilde{\zeta} = \zeta(\tilde{r})$ and $\theta:= \tilde{q} / q >1$, and $C$ is a positive constant depending on $r$ and $\tilde{r}$ only.
\end{lemma}
\begin{proof}
Let $y=(t-s)^{1/2}$. Fix $b$ such that $\frac{q-1}{q}>b>\frac{-\zeta-1}{q}$.
The triangle inequality gives
$$|\gamma(t)-\gamma(s)|^q\lesssim |\gamma(t)-\hat{f}_t(iy)|^q+ |\gamma(s)- \hat{f}_s(iy)|^q + |\hat{f}_t(iy)-\hat{f}_s(iy)|^q.$$
\begin{equation}\label{e:t-s}
 \lesssim |\gamma(t)-\hat{f}_t(iy)|^q+ |\gamma(s)- \hat{f}_s(iy)|^q + |f_t(iy + U_s) - f_s(iy+ U_s)|^q + |f_t(iy+ U_t) - f_t(iy + U_s)|^q
 \end{equation}
We will show that
$$ \E(|\gamma(t)-\hf_t(iy)|^q) \lesssim  t^{-\zeta/2} (t-s)^{(\zeta+q)/2}. $$
Indeed, note that $\gamma(t)=\lim_{u\to 0^+} \hf_t(iu)$ 


$$\begin{array}{rcll}
 |\gamma(t)-\hf_t(iy)|^q&\leq & (\int^y_0 |\hf'_t(iu)|du)^q&\\
 &=&  (\int^y_0 |\hf'_t(iu)|u^b \cdot u^{-b} du)^q\\
&\leq &(  \int^y_0 |\hf'_t(iu)|^q u^{bq} du ) (\int^y_0 u^{-\frac{bq}{q-1}})^{q-1} & \mbox{ by H{\"o}lder's inequality and } q> 1\\
&\lesssim& (  \int^y_0 |\hf'_t(iu)|^q u^{bq} du ) y^{q-1-bq} &\mbox{ since } q-1-bq>0. \\
\end{array}$$
So
$$\begin{array}{rcll}
 \E(|\gamma(t)-\hf_t(iy)|^q) &\lesssim &\left(  \int^y_0 \E(|\hf'_t(iu)|^q) u^{bq} du \right) y^{q-1-bq}\\
&\lesssim &(  \int^y_0 t^{-\zeta/2}u^{\zeta} u^{bq} du ) y^{q-1-bq} & \mbox{ by Lemma \ref{lem:der} and since } t\geq t-s\geq u^2\\
&\lesssim & t^{-\zeta/2} y^{\zeta+bq+1} y^{q-1-bq} &\mbox{ since } \zeta+bq>-1\\
&=&t^{-\zeta/2} (t-s)^{(\zeta+q)/2}.\\
\end{array}$$
In a similar way, we will attain 
\begin{equation}
\E(|\gamma(s)-\hf_s(iy)|^q) \lesssim A_s (t-s)^{(\zeta+q)/2}.
\end{equation}
Next
$$\begin{array}{rcl}
|f_t(iy + U_s) - f_s(iy+ U_s)|^q&\lesssim & \left(y |\hf'_s(iy)|\right)^q ~~~\mbox{ by Lemma \ref{key-lemma2}}.
	\end{array}
$$
Therefore by Lemma \ref{lem:der}
\begin{equation}
\E(|f_t(iy + U_s) - f_s(iy+ U_s)|^q)\lesssim y^q\E( |\hf'_s(iy)|^q)\lesssim A_s (t-s)^{(\zeta+q)/2}.
\end{equation}
Now for the last term in (\ref{e:t-s})
$$
\begin{array}{rcl}
|f_t(iy+ U_t) - f_t(iy + U_s)|^q&\leq & \left(|U_t-U_s| \sup_{w\in [iy, iy+U_s-U_t]} |\hat{f}'_t(w)|\right)^q\\
&\lesssim & \left(|U_t-U_s| |\hat{f}'_t(iy)| ((|U_t-U_s|/y)^2+1)^l\right)^q~~~\mbox{ by Lemma \ref{key-lemma3}}\\
&\lesssim & |\hf'_t(iy)|^q |U_t-U_s|^q( (|U_t-U_s|/y)^{2lq}+1).
\end{array}$$
So let $X=|U_t-U_s|^q( (|U_t-U_s|/y)^{2lq}+1)$, and with $\theta=\tilde{q}/q>1$ for some $\tilde{r}\in (r,r_c)$ one has
$$
\begin{array}{rcl}
\E(|f_t(iy+ U_t) - f_t(iy + U_s)|^q)&\lesssim &  \E\left( |\hf'_t(iy)|^q X\right)\\
&\lesssim & (\E|\hf'_t(iy)|^{q\theta})^{1/\theta} \left( \E X^{\theta^*} \right)^{1/\theta^*},\\
\end{array}
$$
where $\theta^*=\frac{\theta}{\theta-1}$. Now note that
$$\E (X^{\theta*}) \lesssim  y^{q\theta^*}.$$
So
$$\begin{array}{rcl}
\E(|f_t(iy+ U_t) - f_t(iy + U_s)|^q) & \lesssim & y^q (\E|\hf'_t(iy)|^{q\theta})^{1/\theta}\\
&\lesssim &\dfrac{y^{q+\zeta(\tilde{r})/\theta}}{t^{ \tilde{\zeta}/(2\theta)}}= \dfrac{(t-s)^{\frac{1}{2}(q+ \zeta(\tilde{r})/\theta)}}{t^{ \tilde{\zeta}/(2\theta)}}.
\end{array}$$
\end{proof}
\noindent

\section{Besov regularity of SLE}


For each $\delta>0, q\geq 1$ and for each measurable $\phi:[a,b]\to \mathbb{C}$, define its Besov (or fractional Sobolev) semi-norm  as
\begin{equation} \label{equ:sob}
||\phi||_{W^{\delta, q};[a,b]} =\left( \int^b_a\int^b_a \frac{|\phi(t)-\phi(s)|^q}{|t-s|^{1+\delta q}}\,ds\,dt\right)^{1/q}.
\end{equation}
The space of $\phi$ with $||\phi||_{W^{\delta, q};[a,b]} < \infty$ is a Banach-space, denoted by $W^{\delta,q} = W^{\delta,q} ([a,b])$.
%

\noindent
Lemma \ref{l:gts} applies provided $r\in I$ and $\tilde{r}$ $\in
\left( r,r_{c}\right) $. We can then obtain $W^{\delta,q}$-regularity for SLE trace, restricted to some interval $[\varepsilon,1]$, provided we find 
$\delta, q$ such that%
\begin{equation*}
\mathbb{E}\left\Vert \gamma \right\Vert _{W^{\delta ,q};\left[ \varepsilon ,1%
\right] }^{q}=\int_{\varepsilon}^{1}\int_{\varepsilon}^{1}\frac{ \mathbb{E}[  \left\vert \gamma _{t}-\gamma_{s}\right\vert ^{q}]}{\left\vert t-s\right\vert ^{1+\delta q}}dsdt<\infty.
\end{equation*}%
Though our focus is $\varepsilon=0$, the case $\varepsilon > 0$, say  $\varepsilon \in (0,1)$, has noteworthy features and relates to a phase transition at $\kappa=1$. 
Observe also that%
\begin{equation*}
\tilde{\zeta}\rightarrow \zeta ,\tilde{q}\rightarrow q,\theta \rightarrow 1%
\text{ as }\tilde{r}\downarrow r\text{.}
\end{equation*}
(Loosely speaking, by choosing $\tilde{r}$ sufficiently close to $r$, we can work with the limiting values whenever it comes to power-counting arguments.)
Note that, as a consequence of Lemma \ref{lem:I1I2I3},  the condition $r \in I$ implies ${\zeta+q} > 0$ and then, with $q$ positive (in fact, $q= q(r)>1$ by definition of $I_1$), 
\begin{equation}  \exists \delta \in (0, \frac{\zeta+q}{2q} ).  \label{equ:exist_delta}
\end{equation}
To deal with the behaviour at $s,t$ near $0^+$, we further introduce  
\[
J_{1} = J_{1}\left( \kappa \right) :=\left\{ r\in \mathbb{R}:\zeta \left(
r\right) <2\right\}.
\]
\begin{lemma}  \label{lem:J1}
$I \cap J_1 = I$ for $\kappa >1$. 
\end{lemma}
\begin{proof} Obvious from $\zeta \left( r\right) \equiv r- \frac{%
\kappa r^{2}}{8}=-\frac{\kappa }{8}\left( r-\frac{4}{\kappa }\right) ^{2}+\frac{2}{\kappa } \le \frac{2}{\kappa }.$
\end{proof}  


\begin{thm} \label{thm:preBesovA}
Assume $\kappa > 0$. If $r \in I \cap J_1$, $q=q(r)$ and $\delta$ is picked according to \eqref{equ:exist_delta}, then
\[  \mathbb{E}\left\Vert \gamma \right\Vert _{W^{\delta ,q};\left[ 0,1\right] }^{q}  <\infty. 
\]
For $\kappa \in (0,1]$, under the weaker assumption $r \in I$ one still has
\[
\mathbb{E} \left\Vert \gamma \right\Vert_{W^{\delta ,q};\left[ \varepsilon ,1\right] }^{q}<\infty ,\text{ for any } \varepsilon \in (0,1] .
\]
\end{thm}

\begin{figure}
\centering
\begin{subfigure}{.5\textwidth}
  \centering
  \includegraphics[width=.8\linewidth]{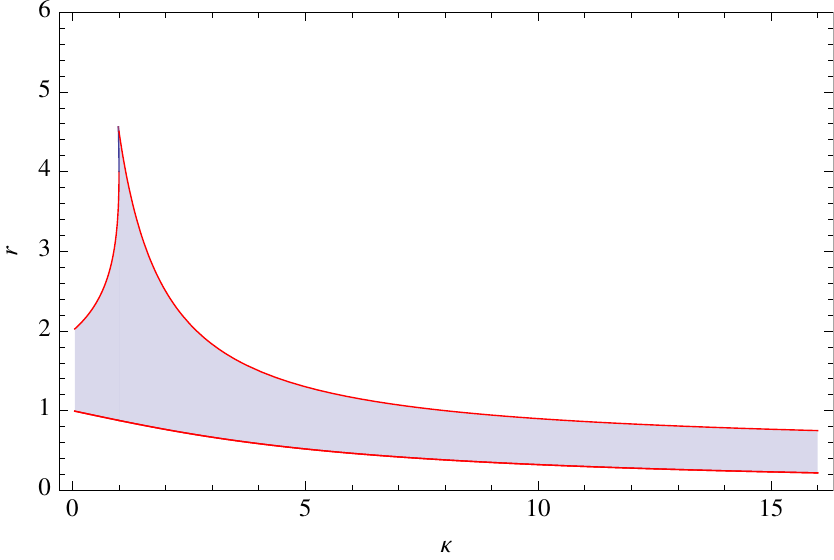}
    \caption{$\kappa \in [0,16]$}
  \label{fig:sub1}
\end{subfigure}%
\begin{subfigure}{.5\textwidth}
  \centering
  \includegraphics[width=.8\linewidth]{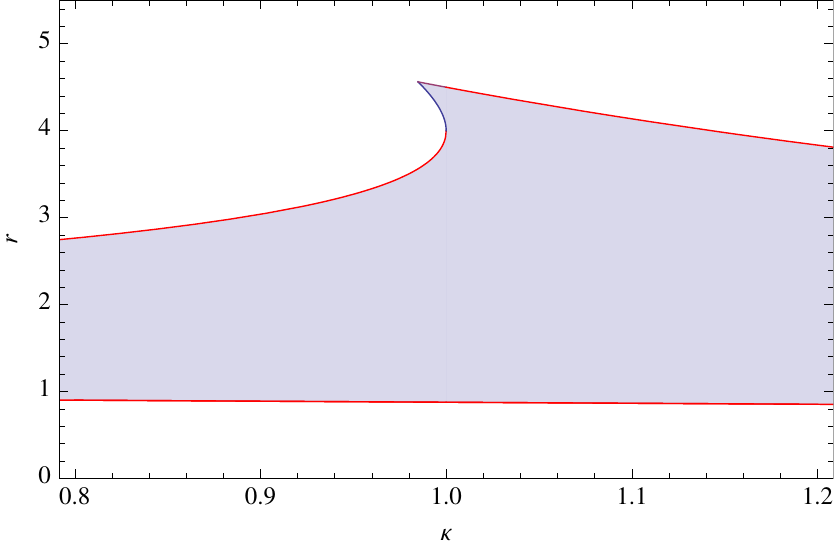}
  \caption{As on the left, zoomed into region $[0.8,1.2]$}
  \label{fig2:sub2}
\end{subfigure}
\caption{Admissible $r$ in the sense of $r \in I \cap J_1$, as function of $\kappa$}
\label{fig:test}
\end{figure}
\begin{remark} The first part of this proposition applies to $\kappa = 8$. \end{remark}
\begin{remark} The conditions on $r$ can be fully spelled out. For instance, when $\kappa >1$, then $r \in I = I \cap J_1 $ iff $ r \in \left( r_{1-},r_{c}\right) =
\left(\kappa^{-1} (4+\kappa  \pm \sqrt{\kappa ^{2}+16}), 1/2 + 4/ \kappa\right)$, cf. Lemma \ref{lem:I1I2I3} above. When $\kappa \le 1$, one takes additionally into account $r \notin [j_{1-},j_{1+}]$ with 
\begin{equation}j_{1\pm} = 4 \kappa^{-1} (1 \pm \sqrt{1-\kappa}),    \label{equ:defj1pm}
\end{equation}
 obtained from 
solving the quadratic inequality $\zeta <2$.
 \end{remark}

\begin{proof}
For $r\in I$ and $\tilde{r}\in \left( r,r_{c}\right) $, Lemma \ref{l:gts} estimates 
$\mathbb{E}\left\vert \gamma _{t}-\gamma _{s}\right\vert ^{q}$ in terms of
factors which are singular at $s=t=0$, these will be controlled thanks to $%
J_{1}$, and factors which are singular at the diagonal $s=t$, to be
controlled via the bound \eqref{equ:exist_delta}.
More specifically, $r\in J_{1}$ guarantees the
integrability of $\max \left( 1,s^{-\zeta /2},t^{-\zeta /2},t^{-\tilde{\zeta}%
/\left( 2\theta \right) }\right)$. Indeed this is obvious for $s^{-\zeta
/2},t^{-\zeta /2}$ since the integrability-at-$0^{+}$-condition ($-\zeta
/2>-1$) is precisely guaranteed by $\zeta \left( r\right) <2$. The same
is true for the exponent $-\tilde{\zeta}/\left( 2\theta \right) =-\tilde{%
\zeta}/\left( 2\tilde{q}/q\right) $, upon choosing $\tilde{r}$ close enough
to $r$. A similar power-counting argument applies to $\left\vert
t-s\right\vert ^{\frac{1}{2}(q+\zeta )}$ resp. $\left\vert t-s\right\vert ^{%
\frac{1}{2}(q+\tilde{\zeta}/\theta )}$. Taking into account factor $%
\left\vert t-s\right\vert ^{1+\delta q}$ which appears in the definition of
the $W^{\delta ,q}$-norm, the integrability-at-diagonal-condition becomes%
\begin{equation*}
\frac{1}{2}(q+\zeta )-\left( 1+\delta q\right) >-1
\end{equation*}%
which is precisely what is guaranteed by \eqref{equ:exist_delta}.
\end{proof}
\newpage
\section{Optimal $p$-variation regularity of SLE}
For each $p\geq 1$, and each continuous function $\phi$ defined on an interval $[a,b]$, we define its $p$-variation as
$$||\phi||_{p-var;[a,b]}=\left(\sup_{\cP} \sum^{\# \cP}_{i=1} |\phi(t_i)-\phi(t_{i-1})|^p\right)^{1/p}$$
where the supremum is taken over partitions $\cP=\{t_0,\cdots, t_n\}$ of $[a,b]$. 

\noindent
We would like to apply the embedding in (\ref{lem:variation}). To do that, define ($J_1$ repeated for the reader's convenience)
\begin{eqnarray*}
J_{1} &:=&J_{1}\left( \kappa \right) :=\left\{ r\in \mathbb{R}:\zeta \left(
r\right) <2\right\}, \\
J_{2} &:=&J_{2}\left( \kappa \right) =\left\{ r\in \mathbb{R}:\zeta \left(
r\right) +q\left( r\right) >2\right\}.
\end{eqnarray*}
Observe that $r\in J_{2}$ is precisely equivalent to  
\begin{equation}  \label{J2delta}
\exists \delta \in \left( \frac{1}{q\left( r\right) },\frac{\zeta \left(
r\right) +q\left( r\right) }{2q\left( r\right) }\right) . 
\end{equation}

Write $(a,b)$ for an open interval (of $\R$), and agree further that $(a,b) = \emptyset$ when $a=b$.

\begin{lemma} \label{lem:J2} (i) $r \in J_2$ iff it is an element in the open interval with endpoints $1, 8/ \kappa$ (and empty for $\kappa = 8$).
 
 \noindent (ii) $r \in I \cap J_2 $ iff $ r \in (1,r_c) \equiv (1, 1/2+ 4 / \kappa)$, in case $\kappa < 8$, and $ r \in (8/\kappa, r_c)$ for $\kappa >8$.
 
 \noindent (iii) With $j_{1\pm}$ as introduced in (\ref{equ:defj1pm}),
 $$I\cap J_1\cap J_2 = \left\{ \begin{array}{lcl}
 (1,j_{1-})\cup (\min(j_{1+},r_c),r_c) &\mbox{ when } &\kappa\in (0,1], \\
 (1,r_c)  & \mbox{ when } & \kappa \in (1,8),\\
 \emptyset & \mbox{ when } & \kappa = 8,\\
 (8/\kappa, r_c) & \mbox{ when } & \kappa\in (8,\infty).
 \end{array}\right.$$
 (See Figure \ref{fig3:sub2}, and also Figure \ref{fig2:sub2} for a zoom, just below $\kappa=1$, where $(\min(j_{1+},r_c),r_c) \ne \emptyset$.)
\end{lemma} 
\begin{proof}
Part (i) and (ii) come from the fact that 
$$\zeta(r)+q(r)-2 = -\frac{\kappa}{4}(r-1)(r-\frac{8}{\kappa}).$$
For part (iii), the case $\kappa\in (0,1]$ follows from the fact that
$$1<j_{1-}<r_c.$$
The other cases are straightforward.
\end{proof}
\begin{thm} \label{thm:Pvar}
If $r \in I \cap J_1 \cap J_2$, $q=q(r)$,  then for all $\delta$ as in \eqref{J2delta} and $p := 1/\delta$, we have
\[  \mathbb{E}\left\Vert \gamma \right\Vert _{p-var;\left[ 0,1\right] }^{q}  <\infty. 
\]
Optimization over the range of admissible $r$, shows that one can take any
  $$
  p > p_* := q(\min (1, 8 / \kappa)) = \min (1 + \kappa / 8, 2).
  $$                
\end{thm}

\begin{figure}
\centering
\begin{subfigure}{.5\textwidth}
  \centering
  \includegraphics[width=.8\linewidth]{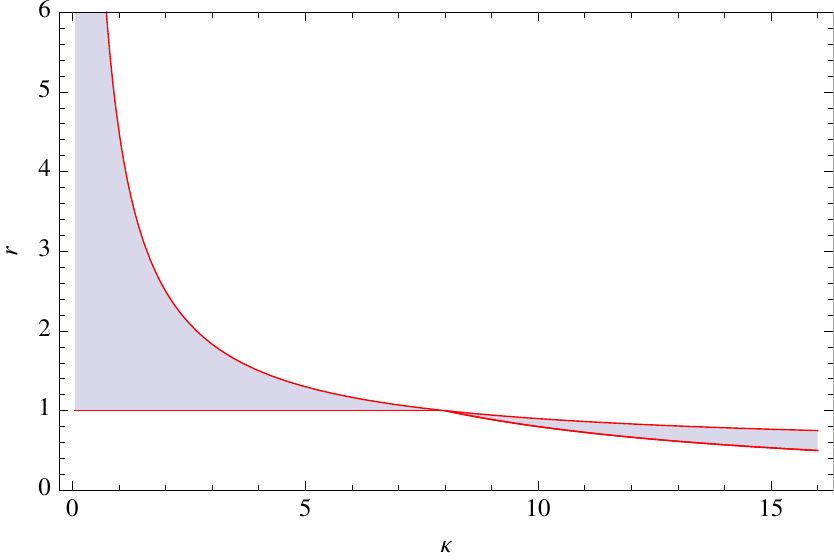}
    \caption{$r \in I \cap J_2$, as function of $\kappa$}
  \label{fig:sub1}
\end{subfigure}%
\begin{subfigure}{.5\textwidth}
  \centering
  \includegraphics[width=.8\linewidth]{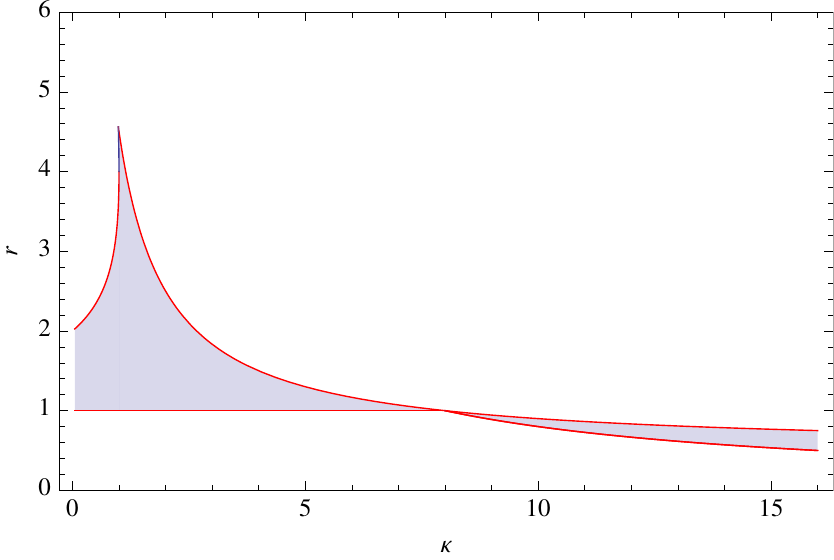}
  \caption{$r \in I \cap J_1 \cap J_2$, as function of $\kappa$}
  \label{fig3:sub2}
\end{subfigure}
\caption{Admissible $r$ for Theorems \ref{thm:Pvar} and \ref{thm:Hoel}. Note $I \cap J_2 = \emptyset$ when $\kappa = 8$.}
\label{fig:test}
\end{figure}

\begin{proof} The first statement follows immediately from (\ref{lem:variation}), Theorem \ref{thm:preBesovA} and by noting that
$$ (\frac{1}{q},\frac{\zeta+q}{2q})\subset (0,1)$$
when $r\in I$. 
The infimum of all possible $p$ is
$$p_*=\inf_{r\in I\cap J_1\cap J_2} \frac{2q}{\zeta+q}.$$
 
\noindent One sees that 
$$\frac{\zeta}{q}=1+\frac{\frac{\ka}{4}}{\frac{\ka}{8}r - (1+\frac{\ka}{4})}$$ 
is a decreasing function and so 
$$\dfrac{2q}{\zeta+q} = \dfrac{2}{\zeta / q + 1}$$
is small when $r$ is small. With Lemma \ref{lem:J2}, it is then easy to see that the ``optimal $r$'' is given by 
\begin{equation} \label{e:rmin}
r_{\min} := \inf ( I \cap J_1 \cap J_2) = \min (1, 8 / \kappa).
\end{equation}
Hence,
when $\ka\in (0,8)$,
$$p_* =  \frac{2q(r)}{\zeta(r) + q(r)}\bigg|_{r=1} = q(r)|_{r=1}= 1 + \frac{\ka}{8}.$$
When $\ka\in (8,\infty)$,
$$p_* = \left. \frac{2q(r)}{\zeta(r) + q(r)}\right\vert_{r=\frac{8}{\ka}} = q(\frac{8}{\ka})=2.$$

\end{proof}

\begin{corollary}[Hausdorff dimension upper bound; \cite{RS05}] \label{cor:dim}
$\dim _{H}\left( \gamma |_{\left[ 0,1\right] }\right) \le \min \left( 1+\kappa /8,2\right).$
\end{corollary}
\begin{proof}
By a property of $p$-variation, the map $\gamma |_{\left[ 0,1\right]
}$ can be reparametrized to a $\delta$-H\"older map $\tilde{\gamma}$, with $\delta=1/p$, so that by basic facts of Hausdorff dimension of sets under H\"older maps,
\begin{equation*}
\dim _{H}\left( \gamma |_{\left[ 0,1\right] }\right) =\dim _{H} (\tilde{%
\gamma}) \leq \frac{1}{\delta }\dim _{H}\left(
[0,1]\right) =p\text{.}
\end{equation*}%
Take $p\downarrow p_{\ast }=\min \left( 1+\kappa /8,2\right) $ to recover
the stated upper bound on the Hausdorff dimension of SLE$_{\kappa }$. \end{proof} 

This upper bound was first derived by Rohde--Schramm \cite{RS05}; equality was later established by Beffara \cite{Be08} 
which in turn shows that our $p$-variation result, any $p>p_*$, is indeed optimal. 


\section{Optimal H{\"o}lder exponent}
Recall that for each $\alpha\in (0,1]$, the $\alpha$-H\"older semi-norm of a continous function $\phi$ defined on an interval $[a,b]$ is
$$||\phi||_{\alpha\text{-H\"ol};\left[a,b\right]} = \sup_{s\neq t\in [a,b]} \frac{|\phi(s)-\phi(t)|}{|s-t|^\alpha}.$$

We follow the same logic as in the previous section, again apply the embedding in (\ref{lem:variation}), which is possible exactly when $r \in I \cap J_1 \cap J_2$. 
As a consequence, we recover the (optimal) SLE H\"older regularity of  \cite[Theorem 1.1]{JVL11},
with the novelty of having some control over moments. 

\begin{thm} \label{thm:Hoel}
If $r \in I \cap J_2$, $q=q(r)$, then for all $\delta$ as in \eqref{J2delta},  and $\alpha := 1/\delta - q$, we have
\[  \mathbb{E}\left\Vert \gamma \right\Vert _{\alpha \text{-H\"ol};\left[ \varepsilon,1\right] }^{q}  <\infty
\]
for any $\varepsilon \in (0,1]$.
Optimization over the range of admissible $r$ shows that one can take any H\"older exponent
\begin{equation*}
\alpha < \alpha _{\ast }\left( \kappa \right) =1-\frac{\kappa }{24+2\kappa -8\sqrt{%
\kappa +8}}.
\end{equation*}
If $r \in I \cap J_1 \cap J_2$, everything else as above, then
\[  \mathbb{E}\left\Vert \gamma \right\Vert _{\alpha \text{-H\"ol};\left[0,1\right] }^{q}  <\infty
\]
and here one can take any H\"older exponent $\alpha < \min (\alpha _{\ast }, 1/2)$.
\end{thm}

\begin{proof}
The statements about finiteness of moments are immediate by Theorem \ref {thm:preBesovA} and the Besov-Holder embedding (\ref{lem:Holder}). We can take any exponent $\alpha<\hat\alpha$, where
$\hat\alpha$ is the supremum of $\frac{1}{\delta}-q$ with $\delta$ as in \eqref{J2delta} and with $r\in I\cap J_2$ or $r\in I\cap J_1\cap J_2$ depending on whether we consider $\left\Vert \gamma \right\Vert _{\alpha \text{-H\"ol};\left[ \varepsilon,1\right] }$ or $\left\Vert \gamma \right\Vert _{\alpha \text{-H\"ol};\left[ 0,1\right] }$. Thus,
$$\hat\alpha = \sup_{r} \frac{\zeta+q-2}{2q}.$$

\begin{figure}
\centering
\begin{subfigure}{.5\textwidth}
  \centering
  \includegraphics[width=.8\linewidth]{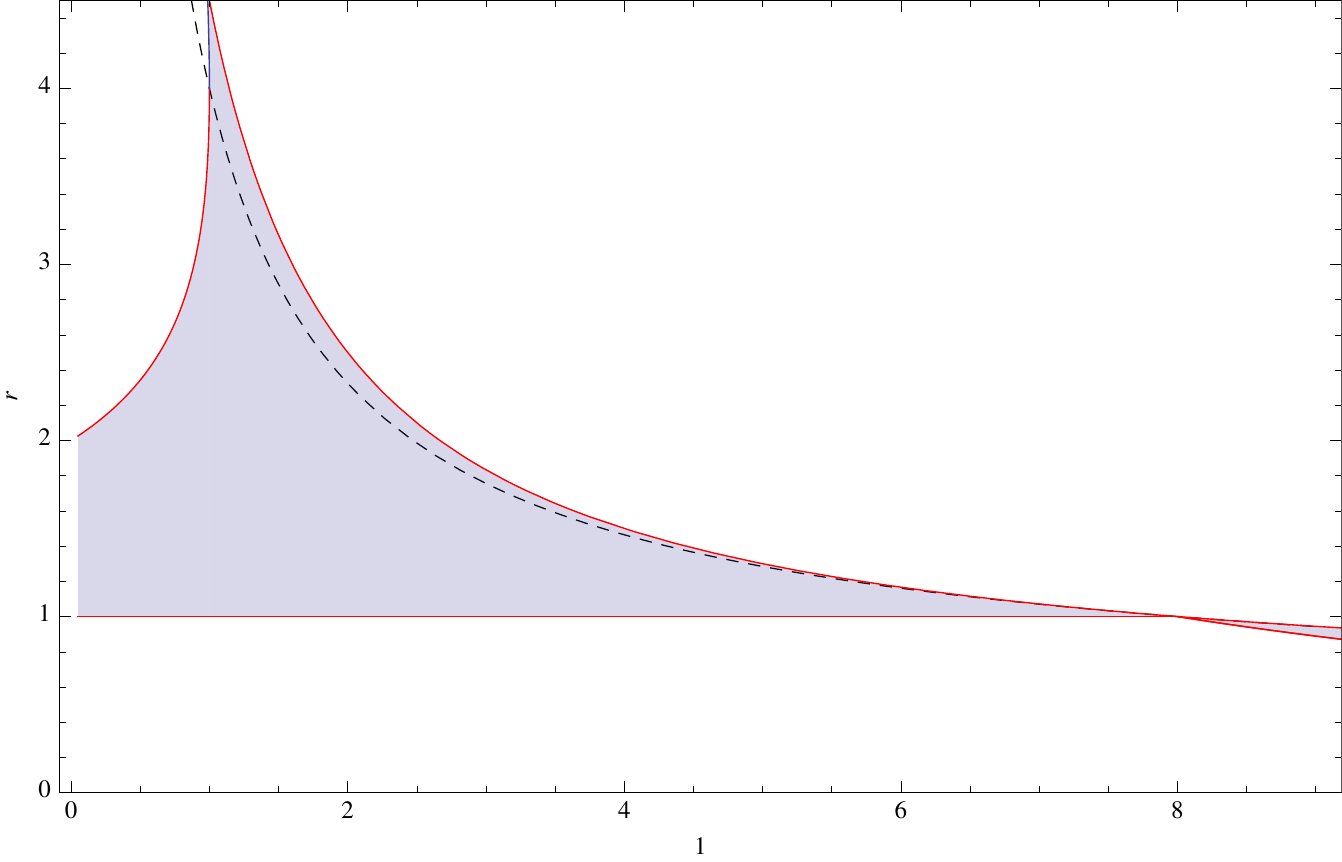}
    \caption{$r \in I \cap J_1 \cap J_2$, with $\kappa \in [0,9]$}
  \label{fig:sub1}
\end{subfigure}%
\begin{subfigure}{.5\textwidth}
  \centering
  \includegraphics[width=.8\linewidth]{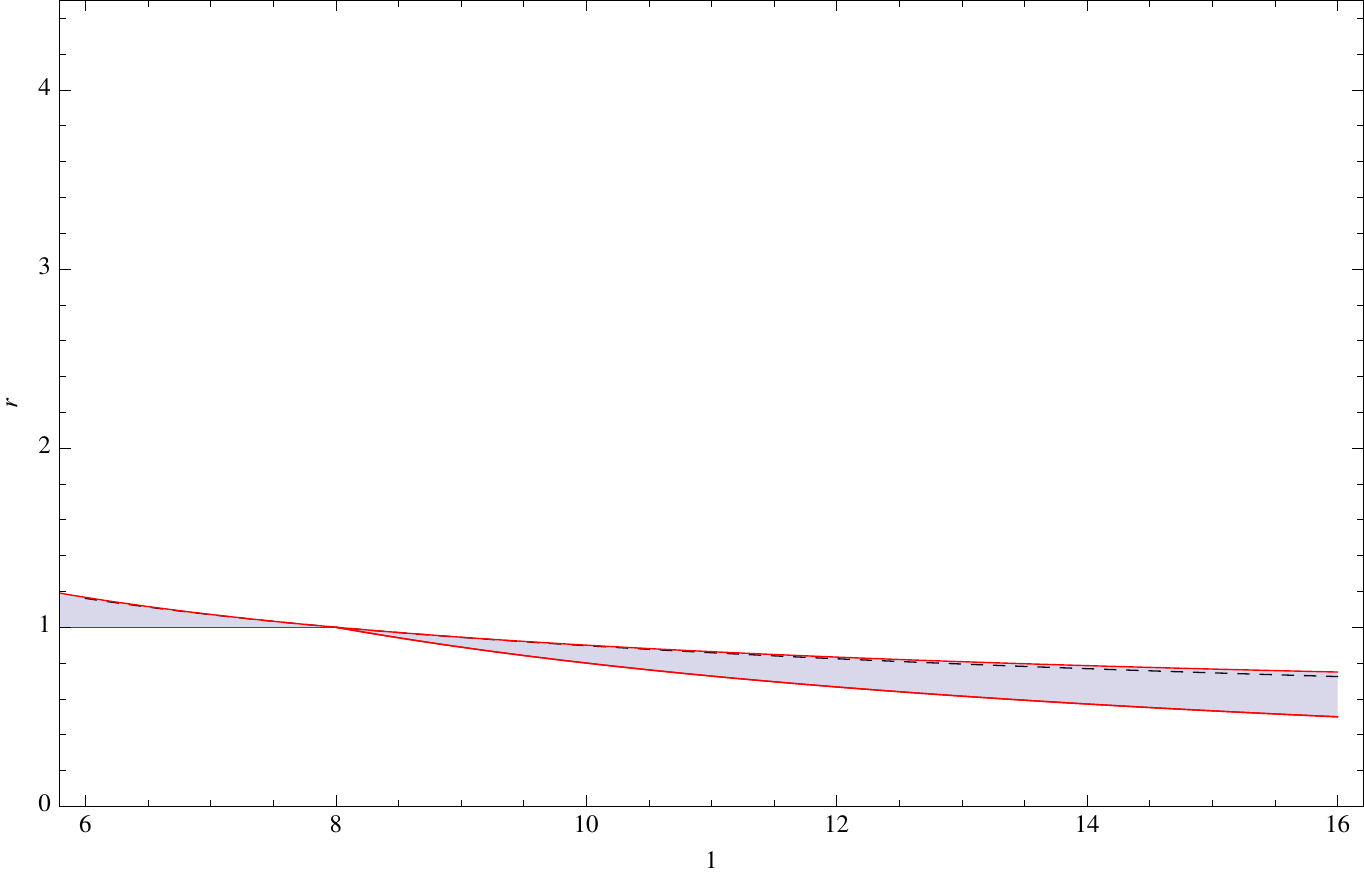}
  \caption{As on the left, $\kappa \in [6,16]$}
  \label{fig:sub2}
\end{subfigure}
\caption{Dashed line for $r=r(\kappa) \in I \cap J_2$ which maximizes H\"older exponent}
\label{fig:test2}
\end{figure}

\noindent
Observe that the function $\phi(r)=\frac{\zeta+q-2}{2q}$ satisfies
$$\phi'(r)\geq 0 \Leftrightarrow r\in [r_-,r_+]$$
where $r_{\pm} = \frac{4(-2\pm \sqrt{8+\kappa})}{\kappa}$.

\noindent
Consider the case $\kappa\in (1,\infty)\backslash\{8\}$. By Lemma  \ref{lem:J2},
$$I\cap J_1\cap J_2=I\cap J_2 = \left\{
\begin{array}{rcl}
(1, r_c) &\mbox{ when } &\kappa <8,\\
(8/\kappa, r_c) &\mbox{ when } &\kappa >8.\\
\end{array}
\right.$$
One can check that $r_-<0$ and $r_+\in I\cap J_2$. Hence
$$\hat\alpha = \phi(r_+) = 1-\frac{\kappa}{24+2\kappa - 8\sqrt{\kappa+8}}.$$

\noindent
Consider the case $\kappa\in (0,1]$. Concerning $\left\Vert \gamma \right\Vert _{\alpha \text{-H\"ol};\left[ \varepsilon,1\right] }$, the conclusion does not change:
$$\sup_{r\in I\cap J_2} \phi(r)= \sup_{r\in (1,r_c)} \phi(r) = \phi(r_+)=\alpha_*(\kappa).$$
Concerning $\left\Vert \gamma \right\Vert _{\alpha \text{-H\"ol};\left[0,1\right] }$, note that
$$1<j_{1-}\leq r_+ < \min (j_{1+},r_c)\leq r_c$$
and that  by Lemma \ref{lem:J2},
$$I\cap J_1\cap J_2 = (1, j_{1-})\cup (\min(j_{1+},r_c),r_c).$$
Therefore,
$$\sup_{r\in I\cap J_1\cap J_2} \phi(r)=\max\{\phi(j_{1-}),\phi(j_{1+})\}=\frac{2+q-2}{2q}=\frac{1}{2}=\min(\alpha_*,1/2).$$
\end{proof}

\section{Further discussion} 

{\bf Quantified finite $q$-moments, $p$-variation case} Fix $  p > p_*= \min (1 + \kappa / 8, 2)$ so that, according to Theorem \ref{thm:Pvar}, there exists $q>1$ so that
\[  \mathbb{E}\left\Vert \gamma \right\Vert _{p-var;\left[ 0,1\right] }^{q}  <\infty. 
\]
How large can we take $q$? Our method allow here to identify a range of finite $q$-moments, with $q \in [1,Q)$ with $Q=Q(p,\kappa)$ .

Since $q$ is strictly increasing, for any $p>p_*$, a possible choice is $Q = Q_* := q(\min(1,8/\kappa)) = \min(1+\kappa/8,2)$. Giving up on pleasant formulae, one can do better. 
Fixing $p>p_*$, we can take 
$$Q=\sup_r q(r)$$
where $r$ satisfies $r\in I\cap J_1\cap J_2$ and that $\frac{2q(r)}{\zeta(r)+q(r)}<p<q(r)$. We let $Q=0$ if there is no such $r$.

Let $\phi(r)=\frac{2q(r)}{\zeta(r)+q(r)}$ and note
\begin{itemize}
\item $\phi(r)$ and $q(r)$ are strictly increasing on $I$,
\item $p_*=\inf_{r\in I\cap J_1\cap J_2} \phi(r) = \phi(r_{\min})$,
\item $\phi(r)<q(r)$.
\end{itemize}

If $p\geq \phi(r_c)$, then $Q=0$. Consider $p\in (\phi(r_{\min}),\phi(r_c))$. There exists $\hat{r}\in (r_{\min},r_c)$ such that $\hat{r} = \sup\{r\in I\cap J_1\cap J_2: \phi(r) < p\}$. Thus,
$$\phi(r)<p \Leftrightarrow r<\hat{r}$$
and, therefore, 
$$Q = \sup_{r\in I\cap J_1\cap J_2:r<\hat{r}, p<q(r)} q(r) = q(\hat{r}).$$

For the value of $\hat{r}$ we have
$$\hat{r}=\left\{\begin{array}{rcl} j_{1-} &\mbox{ when } & \kappa \in (0,1] \mbox{ and } p\in \phi(I\cap J_1\cap J_2) \\
\phi^{-1}(p) =\frac{(8+\kappa) p - (8+2\kappa)}{\kappa(p-1)} & & \mbox{ otherwise.} 
\end{array}\right.$$
 \noindent
Note that $Q=Q(p,\kappa) > Q_* = Q_*(\kappa)$. On the other hand, as $p \downarrow p_*$, $\hat{r}$ approaches $r_{\min}$, so that $Q \to Q_*$.
 
%
%
 \noindent
(A similiar discussion about $q$-moments for the $\alpha$-H\"older case is left to the reader.)


\medskip \medskip

\noindent
{\bf Beyond H\"older and variation}
At last, we note that it is possible to regard H\"older and variation regularity as extreme points of a scale of Riesz type variation spaces, recently related to a scale of Nikolskii spaces, see \cite{2016arXiv160903132F}. As
$W^{\delta,p}$ embedds into these spaces, this would allow for another family of SLE regularity statements. 

\bibliographystyle{alpha}
\bibliography{refs}

\end{document}